\documentclass[11pt]{article}
\usepackage{latexsym,amssymb,amsmath,amsfonts,amsthm}
\usepackage{mathrsfs}
\usepackage{subfigure}
\usepackage{epsf,graphicx,epsfig,color,cite,cases}
\usepackage{subfigure,graphics,multirow,marginnote,enumerate,bm}
\newcommand{\R}{{\mat R}}

\newcommand{\C}{{\mat C}}
\newcommand{\E}{{\mat E}}

\newcommand{\ds}{\displaystyle}
\newcommand{\no}{\nonumber}
\newcommand{\be}{\begin{equation}\begin{aligned}}
\newcommand{\ben}{\begin{equation*}\begin{aligned}}
\newcommand{\en}{\end{aligned}\end{equation}}
\newcommand{\enn}{\end{aligned}\end{equation*}}

\newcommand{\ba}{\backslash}
\newcommand{\pa}{\partial}

\newcommand{\ov}{\overline}

\newcommand{\Rt}{{\rm Re}}
\newcommand{\curl}{{\rm curl}}
\newcommand{\Curl}{{\rm Curl}}

\newcommand{\Dive}{{\rm Div}}
\newcommand{\g}{\gamma}
\newcommand{\G}{\Gamma}

\newcommand{\vep}{\varepsilon}

\newcommand{\Om}{\Omega}
\newcommand{\om}{\omega}
\newcommand{\sig}{\sigma}
\newcommand{\Sig}{\Sigma}

\newcommand{\al}{\alpha}

\newcommand{\wih}{\widehat}
\newcommand{\ti}{\times}

\newcommand{\wit}{\widetilde}

\newcommand{\ra}{\rightarrow}

\newcommand{\na}{\nabla}
\newcommand{\mat}{\mathbb}
\newcommand{\se}{\setminus}
\newcommand{\ify}{\infty}

\newcommand{\la}{\lambda}

\newcommand{\les}{\lesssim}
\newcommand{\ch}{\check}
\newcommand{\V}{\Vert}
\newcommand{\diag}{{\rm diag}}

\newcommand{\n}{\bm{n}}

\newcommand{\0}{\bm{0}}

\newcommand{\on}{{\rm on}}
\newcommand{\gin}{{\rm in}}
\newcommand{\PML}{{\rm PML}}
\newtheorem{theorem}{Theorem}[section]
\newtheorem{lemma}[theorem]{Lemma}

\newtheorem{remark}[theorem]{Remark}

\begin{document}
\renewcommand{\theequation}{\arabic{section}.\arabic{equation}}

\title{\bf Convergence of the uniaxial PML method for time-domain electromagnetic scattering problems}
\author{Changkun Wei\thanks{Research Institute of Mathematics,
Seoul National University, Seoul, 08826, Republic of Korea ({\tt ckun.wei@snu.ac.kr})}
\and
Jiaqing Yang\thanks{School of Mathematics and Statistics, Xi'an Jiaotong University,
Xi'an, Shaanxi, 710049, China ({\tt jiaq.yang@xjtu.edu.cn})}
\and
Bo Zhang\thanks{NCMIS, LSEC and Academy of Mathematics and Systems Science, Chinese Academy of Sciences,
Beijing, 100190, China and School of Mathematical Sciences, University of Chinese Academy of Sciences,
Beijing 100049, China ({\tt b.zhang@amt.ac.cn})}
}
\date{}
\maketitle


\begin{abstract}
In this paper, we propose and study the uniaxial perfectly matched layer (PML) method
for three-dimensional time-domain electromagnetic scattering problems, which has a great advantage
over the spherical one in dealing with problems involving anisotropic scatterers.
The truncated uniaxial PML problem is proved to be well-posed and stable, based on the Laplace transform technique
and the energy method. Moreover, the $L^2$-norm and $L^{\ify}$-norm error estimates in time are given
between the solutions of the original scattering problem and the truncated PML problem, leading to
the exponential convergence of the time-domain uniaxial PML method in terms of the thickness and absorbing
parameters of the PML layer. The proof depends on the error analysis between the EtM operators for the original
scattering problem and the truncated PML problem, which is different from our previous work
(SIAM J. Numer. Anal. 58(3) (2020), 1918-1940).

\vspace{.2in}
\end{abstract}

\section{Introduction}\label{intro}
\setcounter{equation}{0}

This paper is concerned with the time-domain electromagnetic scattering by a perfectly conducting obstacle
which is modeled by the exterior boundary value problem:
\begin{subnumcases}{}\label{maxwell1}
  \ds\na\times\bm{E} +\mu\pa_t\bm{H}=\0 & $\gin\;\;\;(\R^3\ba\ov{\Om})\times(0,T)$,\\
  \label{maxwell2}
  \ds\na\times\bm{H} -\vep\pa_t\bm{E} =\bm J & $\gin\;\;\;(\R^3\ba\ov{\Om})\times(0,T)$,\\
  \label{boundary}
  \ds\n\times\bm{E} =\0 & $\on\;\;\;\G\times(0,T)$,\\
  \label{ititial}
  \ds \bm E(x,0) =\bm H(x,0) = \0 & $\gin\;\;\;\R^3 \ba \ov{\Om}$,\\
  \label{smrc}
  \ds\hat{x}\times\left(\pa_t\bm E\times\hat{x}\right)+\hat{x}\times\pa_t\bm{H}
  = o\left(|x|^{-1}\right)\;\; & ${\rm as}\;\;\;|x|\ra\infty,\;\;t\in(0,T)$.
\end{subnumcases}
Here, $\bm E$ and $\bm H$ denote the electric and magnetic fields, respectively, $\Om\subset\R^3$
is a bounded Lipschitz domain with boundary $\G$ and $\n$ is the unit outer normal vector to $\G$.
Throughout this paper, the electric permittivity $\vep$ and the magnetic permeability $\mu$ are
assumed to be positive constants. Equation \eqref{smrc} is the well-known Silver-M\"{u}ller
radiation condition in the time domain with $\hat{x}:={x}/{|x|}$.

Time-domain scattering problems have been widely studied recently due to their capability of
capturing wide-band signals and modeling more general materials and nonlinearity,
including their mathematical analysis
(see, e.g., \cite{BGL2018,Chen2014,LLA2015,GL2016,GL2017,GLZ2017,Hsiao2015,wy2019,wyz-amas-2020}
and the references quoted there).
The well-posedness and stability of solutions to the problem \eqref{maxwell1}-\eqref{smrc}
have been proved in \cite{chen2008} by employing an exact transparent boundary condition (TBC)
on a large sphere. Recently, a spherical PML method has been proposed in \cite{wyz-sinum-2020} to solve
the problem \eqref{maxwell1}-\eqref{smrc} efficiently, based on the real coordinate stretching technique
associated with $[\Rt(s)]^{-1}$ in the Laplace transform domain with the Laplace transform variable
$s\in\C_+:=\{s=s_1+is_2\in\C:\;s_1>0,\;s_2\in\R\}$, and its exponential convergence has also been established
in terms of the thickness and absorbing parameters of the PML layer.

In this paper, we continue our previous study in \cite{wyz-sinum-2020} and propose and study
the uniaxial PML method for the problem \eqref{maxwell1}-\eqref{smrc}, based on the real coordinate
stretching technique introduced in \cite{wyz-sinum-2020}, which uses a cubic domain to define the PML problem
and thus is of great advantage over the spherical one in dealing with problems involving anisotropic scatterers.
We first establish the existence, uniqueness and stability estimates of the PML problem by
the Laplace transform technique and the energy argument and then prove the exponential convergence in both
the $L^2$-norm and the $L^{\ify}$-norm in time of the time-domain uniaxial PML method.
Our proof for the $L^2$-norm convergence follows naturally from the error estimate between the EtM operators
for the original scattering problem and its truncated PML problem established also in the paper,
which is different from \cite{wyz-sinum-2020}.
The $L^{\ify}$-norm convergence is obtained directly from the time-domain variational formulation
of the original scattering problem and its truncated PML problem with using special test functions.

The PML method was first introduced in the pioneering work \cite{Berenger1994} of B\'erenger in 1994
for efficiently solving the time-dependent Maxwell's equations. Its idea is to surround the computational
domain with a specially designed medium layer of finite thickness in which the scattered waves decay rapidly
regardless of the wave incident angle, thereby greatly reducing the computational complexity of the scattering
problem. Since then, various PML methods have been developed and studied in the literature
(see, e.g., \cite{Bramble2007,chew1994,Collino1998,DeHoop2001,Hagstrom1999,TC2001,DJ2006}
and the references quoted there). Convergence analysis of the PML method has also been widely studied for
time-harmonic acoustic, electromagnetic, and elastic wave scattering problems.
For example, the exponential convergence has been established in terms of the thickness of the PML layer
in \cite{Lassas1998,HSZ2003,BW2005,CL2005,Bramble2007,CC2007,BPT2010,CZ2017} for the circular or spherical
PML method and in \cite{CW2008,BP2008,CZ2010,BP2012,BP2013,CXZ2016,CCZ2013} for the uniaxial (or Cartesian)
PML method. Among them, the proof in \cite{BW2005} is based on the error estimate between
the electric-to-magnetic (EtM) operators for the original electromagnetic scattering problem and
its truncated PML problem, while the key ingredient of the proof in \cite{CC2007} and \cite{CCZ2013}
is the decay property of the PML extensions defined by the series solution and the integral representation
solution, respectively.
On the other hand, there are also several works on convergence analysis of the time-domain PML method for
transient scattering problems. For two-dimensional transient acoustic scattering problems, the exponential
convergence was proved in \cite{chen2009} for the circular PML method and in \cite{CW2012} for the uniaxial
PML method, based on the complex coordinate stretching technique.
For the 3D time-domain electromagnetic scattering problem \eqref{maxwell1}-\eqref{smrc},
the spherical PML method was proposed in \cite{wyz-sinum-2020} based on the real coordinate stretching technique
associated with $[\Rt(s)]^{-1}$ in the Laplace transform domain with the Laplace transform variable $s\in\C_+$,
and its exponential convergence was established by means of the energy argument and the exponential decay
estimates of the stretched dyadic Green's function for the Maxwell equations in the free space.
In addition, we refer to \cite{BGL2018} for the well-posedness and stability estimates of the time-domain PML method
for the two-dimensional acoustic-elastic interaction problem, and to \cite{wyz-cms-2020} for the convergence
analysis of the PML method for the fluid-solid interaction problem above an unbounded rough surface.

The remaining part of this paper is as follows. In Section \ref{sec_fs}, we introduce some basic Sobolev spaces
needed in this paper. In Section \ref{sec_wp}, the well-posedness of the time-domain electromagnetic scattering
problem is presented, and some important properties are given for the transparent boundary condition (TBC)
in the Cartesian coordinate.
In Section \ref{sec_upml}, we propose the uniaxial PML method in the Cartesian coordinate,
study the well-posedness of the truncated PML problem and establish its exponential convergence.
Some conclusions are given in Section \ref{sec_conclusion}.

\section{Functional spaces}\label{sec_fs}

We briefly introduce the Sobolev space $H(\curl,\cdot)$ and its related trace spaces which are used
in this paper. For a bounded domain $D\subset\R^3$ with Lipschitz continuous boundary $\Sigma$,
the Sobolev space $H(\curl,D)$ is defined by
\ben
H(\curl,D):=\{\bm{u}\in L^2(D)^3:\;\na\times\bm{u}\in L^2(D)^3\}
\enn
which is a Hilbert space equipped with the norm
\ben
\|\bm{u}\|_{H(\curl,D)}=\left(\|\bm{u}\|^2_{L^2(D)^3}+\|\na\times\bm{u}\|^2_{L^2(D)^3}\right)^{1/2}.
\enn

Denote by $\bm{u}_{\Sig}=\n\times(\bm{u}\times\bm{n})|_\Sig$ the tangential component of $\bm u$
on $\Sig$, where $\bm n$ denotes the unit outward normal vector on $\Sig$.
By \cite{BCS2002} we have the following bounded and surjective trace operators:
\ben
& \g:\;H^1(D)\ra H^{1/2}(\Sig),\quad \g\varphi = \varphi\quad\on\;\;\Sig,\\
& \g_t:\;H(\curl,D)\ra H^{-1/2}(\Dive,\Sig),\quad\g_t\bm{u}=\bm{u}\times\n\quad\on\;\;\Sig,\\
& \g_T:\;H(\curl,D)\ra H^{-1/2}(\Curl,\Sig),\quad\g_T\bm{u}=\n\times(\bm{u}\times\n)\quad\on\;\;\Sig,
\enn
where $\g_t$ and $\g_T$ are known as the tangential trace and tangential components trace operators,
and $\Dive$ and $\Curl$ denote the surface divergence and surface scalar curl operators,
respectively (for the detailed definition of $H^{-1/2}(\Dive,\Sig)$ and $H^{-1/2}(\Curl,\Sig)$,
we refer to \cite{BCS2002}). By \cite{BCS2002} again we know that $H^{-1/2}(\Dive,\Sig)$ and
$H^{-1/2}(\Curl,\Sig)$ form a dual pairing satisfying the integration by parts formula
\be\label{curl_form}
(\bm{u},\na\times\bm{v})_{D}-(\na\times\bm{u},\bm{v})_{D}
=\left\langle\g_{t}\bm{u},\g_{T}\bm{v}\right\rangle_{\Sig}\quad\forall\;\bm{u},\bm{v}\in\bm{H}(\curl,D),
\en
where $(\cdot,\cdot)_{D}$ and $\langle\cdot,\cdot\rangle_{\Sig}$ denote the $L^2$-inner product
on $D$ and the dual product between $H^{-1/2}(\Dive,\Sig)$ and $H^{-1/2}(\Curl,\Sig)$, respectively.

For any $S\subset\Sig$, the subspace with zero tangential trace on $S$ is denoted as
\ben
H_S(\curl,D):=\left\{\bm{u} \in H(\curl,D):\g_{t}\bm{u} = 0\;\;\on\;\;S\right\}.
\enn
In particular, if $S=\Sig$ then we write $H_{0}(\curl,D):=H_{\Sig}(\curl,D)$.

\section{The well-posedness of the scattering problem}\label{sec_wp}

Let $\Om$ be contained in the interior of the cuboid $B_1:=\{x=(x_1,x_2,x_3)^{\top}\in\R^3:|x_j|<L_j/2,j=1,2,3\}$
with boundary $\G_1=\pa B_1$. Denote by $\n_1$ the unit outward normal to $\G_1$.
The computational domain $B_1\ba\ov{\Om}$ is denoted by $\Om_1$. In this section, we assume that
the current density $\bm J$ is compactly supported in $B_1$ with
\be\label{assump_upml}
\bm J\in H^{10}(0,T;L^2(\Om_1)^3),\;\;\;\pa_t^j\bm J|_{t=0}=0,\;\;j=0,1,2,3,\dots9
\en
and that $\bm J$ is extended so that
\be\label{assump1_upml}
\bm J\in H^{10}(0,\infty;L^2(\Om_1)^3),\;\;\|\bm J\|_{H^{10}(0,\infty;L^2(\Om_1)^3)}
\le C\|\bm J\|_{H^{10}(0,T;L^2(\Om_1)^3)}.
\en

Define the following time-domain transparent boundary condition (TBC) on $\G_1$:
\be\label{tbc}
\mathscr{T}[\bm{E}_{\G_1}]=\bm{H}\times\n_1\quad\;\on\;\;\;\G_1\times(0,T)
\en
which is essentially an electric-to-magnetic (EtM) Calder\'on operator.
Then the original scattering problem \eqref{maxwell1}-\eqref{smrc} can be equivalently
reduced into the initial boundary value problem in a bounded domain $\Om_1\times(0,T)$:
\begin{equation}\label{reduced}
  \begin{cases}
    \na\times\bm{E} +\mu\pa_t\bm{H} = \0 & \gin\;\;\;\Om_1\times(0,T),\\
    \na\times\bm{H} -\vep\pa_t\bm{E} = \bm J & \gin\;\;\;\Om_1\times(0,T),\\
    \n\times\bm{E} =\0 & \on\;\;\;\G\times(0,T),\\
    \bm E(x,0)=\bm H(x,0) = \0 & \gin\;\;\;\Om_1,\\
    \mathscr{T}[\bm E_{\G_1}]=\bm H\times\n_1 & \on\;\;\;\G_1\times(0,T).
  \end{cases}
\end{equation}
The well-posedness of the original scattering problem \eqref{maxwell1}-\eqref{smrc} has been established
in \cite{chen2008} by using the transparent boundary condition on a sphere. Thus the problem \eqref{reduced}
is also well-posed since it is equivalent to the problem \eqref{maxwell1}-\eqref{smrc}.
However, for convenience of the subsequent use in the following sections, we study the problem \eqref{reduced}
directly by studying the property of the EtM operator $\mathscr{T}$.
For any $s\in\C_+:=\{s=s_1+is_2\in\C:\; s_1>0,\;s_2\in\R\}$ let
\ben
\ch{\bm{E}}(x,s) &= \mathscr{L}(\bm{E})(x,s)=\int_0^{\infty} e^{-st}\bm{E}(x,t) dt, \\
\ch{\bm{H}}(x,s) &= \mathscr{L}(\bm{H})(x,s)=\int_0^{\infty} e^{-st}\bm{H}(x,t) dt
\enn
be the Laplace transform of $\bm{E}$ and $\bm{H}$ with respect to time $t$, respectively
(for extensive studies on the Laplace transform, the reader is referred to \cite{Cohen2007}).
Let $\mathscr{B}: H^{-1/2}(\Curl,\G_1)\to H^{-1/2}(\Dive,\G_1)$ be the EtM operator
\be\label{s_etm}
\mathscr{B}[\ch{\bm{E}}_{\G_1}]=\ch{\bm{H}}\times\n_1\quad\on\;\;\;\G_1,
\en
where $\ch{\bm{E}}$ and $\ch{\bm{H}}$ satisfy the exterior Maxwell's equation in the Laplace domain
\be\label{s_exterior_u}
\begin{cases}
\na\times\ch{\bm{E}}+\mu s\ch{\bm{H}} =\0 \quad & \gin\;\;\;\R^3\ba\ov{B}_1,\\
\na\times\ch{\bm{H}}-\vep s\ch{\bm{E}}=\0 \quad & \gin\;\;\;\R^3\ba\ov{B}_1,\\
\hat{x}\times(\ch{\bm{E}}\times\hat{x})+\hat{x}\times\ch{\bm{H}}
=o\left(\frac{1}{|x|}\right)\quad & {\rm as} \;\;\; |x|\ra\infty.
\end{cases}
\en
It is obvious that $\mathscr{T}=\mathscr{L}^{-1}\circ\mathscr{B}\circ\mathscr{L}$.
For each $s\in\C_+$ it is known that, by the Lax-Milgram theorem the problem \eqref{s_exterior_u}
has a unique solution $(\ch{\bm{E}},\ch{\bm{H}})\in H(\curl,\R^3\ba\ov{B}_1)$ .
Thus the operator $\mathscr{B}$ is a well-defined, continuous linear operator.

\begin{lemma}\label{lem_etm}
For each $s\in \C_+$, $\mathscr{B}: H^{-1/2}(\Curl,\G_1)\to H^{-1/2}(\Dive,\G_1)$ is bounded
with the estimate
\be\label{bound_etm}
\|\mathscr{B}\|_{L(H^{-1/2}(\Curl,\G_1),H^{-1/2}(\Dive,\G_1))}\les |s|^{-1} + |s|,
\en
where $L(X,Y)$ denotes the standard space of bounded linear operators from the Hilbert space $X$
to the Hilbert space $Y$. Further, we have
\be\label{positive_dtn}
\Rt\langle\mathscr{B}\bm\om,\bm\om\rangle_{\G_1}\geq 0\;\;\;
\text{for any}\;\;\bm\om\in H^{-1/2}(\Curl,\G_1),
\en
where $\langle\cdot\rangle_{\G_1}$ denotes the dual product between $H^{-1/2}(\Dive,\G_1)$
and $H^{-1/2}(\Curl,\G_1)$.
\end{lemma}

\begin{proof}
First, eliminating $\ch{\bm{H}}$ from \eqref{s_exterior_u} and multiplying both sides of the resulting equation
with $\ov{\bm V}\in H(\curl,\R^3\ba\ov{B}_1)$ yield
\ben
\left|\langle\mathscr{B}[\ch{\bm{E}}_{\G_1}],\g_T{\bm V}\rangle_{\G_1}\right|
&=\left|\int_{\R^3\ba\ov{B}_1}\left[(\mu s)^{-1}\na\times\ch{\bm{E}}\cdot\na\times\ov{\bm{V}}
+\vep s\ch{\bm{E}}\cdot\ov{\bm{V}}dx\right]\right|\\
&\les(|s|^{-1}+|s|)\|\ch{\bm E}\|_{H(\curl,\R^3\ba\ov{B}_1)}\|{\bm V}\|_{H(\curl,\R^3\ba\ov{B}_1)},
\enn
which implies \eqref{bound_etm}.

Now, for any $\bm\om\in H^{-1/2}(\Curl,\G_1)$ suppose $(\ch{\bm E},\ch{\bm H})$ is the solution
to the problem \eqref{s_exterior_u} satisfying the boundary condition $\g_T\ch{\bm E}=\bm{\om}$
on $\G_1$. Let $B_R :=\{x\in\R^3: |x|<R\}$ contain the domain $B_1$.
Eliminating $\ch{\bm H}$ from \eqref{s_exterior_u} and integrating by parts the resulting equation
multiplied with $\ov{\ch{\bm E}}$ over $B_R\ba\ov{B}_1$, we obtain that
\be\label{integral}
&\int_{B_R\ba\ov{B}_1}\left((\mu s)^{-1}|\na\times\ch{\bm E}|^2+\vep s|\ch{\bm E}|^2\right)dx
-\langle\mathscr{B}\bm\om,\bm\om\rangle_{\G_1} \\
&\qquad\qquad+\int_{\pa B_R}\hat{x}\times(\mu s)^{-1}\na\times\ch{\bm E}\cdot\ov{\ch{\bm E}}d\g= 0.
\en
Taking the real part of \eqref{integral} and noting that
\ben
&\left|\hat{x}\times(\ch{\bm E}\times\hat{x})-\hat{x}\times(\mu s)^{-1}\na\times\ch{\bm E}\right|^2 \\
&\qquad=|\hat{x}\times(\ch{\bm E}\times\hat{x})|^2+|\hat{x}\times(\mu s)^{-1}\na\times\ch{\bm E}|^2
-2\Rt(\hat{x}\times(\mu s)^{-1}\na\times\ch{\bm E})\cdot\ov{\ch{\bm E}},
\enn
we have
\be\label{norm_eq}
&\frac{s_1}{\mu|s|^2}\|\na\times \ch{\bm E}\|^2_{L^2(B_R\ba \ov{B}_1)^3}
+\vep s_1\|\ch{\bm E}\|^2_{L^2(B_R\ba \ov{B}_1)^3}-\Rt\langle\mathscr{B}\bm\om,\bm\om\rangle_{\G_1}\\
&\qquad\;\;\;+\frac{1}{2}\left\|\hat{x}\times(\ch{\bm E}\times\hat{x})\right\|_{L^2(\pa B_R)^3}^2
+ \frac{1}{2}\left\|\hat{x}\times(\mu s)^{-1}\na\times\ch{\bm E}\right\|_{L^2(\pa B_R)^3}^2 \\
&\qquad=\frac{1}{2}\left\|\hat{x}\times(\ch{\bm E}\times\hat{x})-\hat{x}\times(\mu s)^{-1}
\na\times\ch{\bm E}\right\|_{L^2(\pa B_R)^3}^2.
\en
By the Silver-M\"{u}ller radiation condition \eqref{smrc} in the $s$-domain, it is known that the right-hand side
of \eqref{norm_eq} tends to zero as $R\ra\infty$.
This implies that $\Rt \langle\mathscr{B}\bm\om,\bm\om\rangle_{\G_1}\ge 0$.
The proof is thus complete.
\end{proof}

By using Lemma \ref{lem_etm} and \cite[Lemmas 4.5-4.6]{wyz-amas-2020}, the time-domain EtM operator $\mathscr{T}$
has the following positive properties which will be used in the error analysis of the time-domain PML solution.

\begin{lemma}\label{lem_positive}
Given $\xi\geq 0$ and $\bm\om(\cdot,t)\in L^2(0,\xi;H^{-1/2}(\Curl,\G_1))$ it holds that
\ben
\Rt\int_{\G_1}\int_0^\xi\left(\int_0^t\mathscr{C}[\bm\om](x,\tau)d\tau\right)\ov{\bm\om}(x,t)dtd\g\geq 0,
\enn
where $\mathscr{C}=\mathscr{L}^{-1}\circ s\mathscr{B}\circ \mathscr{L}$.
\end{lemma}

\begin{lemma}\label{lem_positive1}
Given $\xi\ge 0$ and $\bm\om(\cdot,t)\in L^2(0,\xi;H^{-1/2}(\Curl,\G_1))$ with $\bm\om(\cdot,0)=\0$, it holds that
\ben
\Rt\int_{\G_1}\int_0^\xi\left(\int_0^t\mathscr{C}[\pa_{\tau}{\bm\om}](x,\tau)d\tau\right)
\pa_{\tau}\ov{\bm\om}(x,t)dtd\g\ge 0.
\enn
\end{lemma}

We now introduce the equivalent variational formulation in the Laplace transform domain to the problem \eqref{reduced}.
To this end, eliminate the magnetic field $\bm{H}$ and take the Laplace transform of \eqref{reduced} to get
\begin{equation}\label{s_reduced}
\begin{cases}
\na\times[(\mu s)^{-1}\na\times\ch{\bm{E}}]+\vep s\ch{\bm E} = -\ch{\bm J} & \gin\;\;\;\Om_1,\\
\n\times\ch{\bm{E}} = \0 & \on\;\;\;\G,\\
\mathscr{B}[\ch{\bm E}_{\G_1}] = -(\mu s)^{-1}\na\times\ch{\bm{E}}\times\n_1 & \on\;\;\;\G_1.
\end{cases}
\end{equation}
The variational formulation of \eqref{s_reduced} is then as follows: find a solution
$\ch{\bm E}\in H_{\G}(\curl,\Om_1)$ such that
\be\label{a}
a(\ch{\bm{E}},\bm{V})=-\int_{\Om_1}\ch{\bm J}\cdot\ov{\bm{V}}dx,\quad\forall\;\bm V\in H_{\G}(\curl,\Om_1),
\en
where the sesquilinear form $a(\cdot,\cdot)$ is defined as
\be\label{aform}
a(\ch{\bm{E}},\bm{V})=\int_{\Om_1}\left[(s\mu)^{-1}(\na\times\ch{\bm{E}})\cdot(\na\times\ov{\bm V})dx
+\vep s\ch{\bm{E}}\cdot\ov{\bm V}\right]dx +\langle\mathscr{B}[\ch{\bm E}_{\G_1}],\bm V_{\G_1}\rangle_{\G_1}.
\en
By Lemma \ref{lem_etm} it is easy to see that $a(\cdot,\cdot)$ is uniformly coercive, that is,
\begin{align}
\Rt[a(\ch{\bm{E}},\ch{\bm{E}})]& \gtrsim\frac{s_1}{|s|^2}(\|\na\times\ch{\bm E}\|_{L^2(\Om_1)^3}^2
  +\|s\ch{\bm E}\|_{L^2(\Om_1)^3}^2) \no \\ \label{coer_a}
& \ge s_1\min\{|s|^{-2},1\}\|\ch{\bm{E}}\|^2_{H(\curl,\Om_1)}.
\end{align}
Then, by the Lax-Milgram theorem the problem \eqref{s_reduced} is well-posed for each $s\in\C_+$.
Thus, and by the energy argument in conjunction with the inversion theorem of the Laplace transform
(cf. \cite{chen2008}) the well-posedness of the problem \eqref{reduced} follows.
In particular, $\mathscr{T}[\bm E_{\G_1}]\in L^2\left(0,T;H^{-1/2}(\Dive,\G_1)\right)$.

\section{The uniaxial PML method}\label{sec_upml}

In practical applications, the scattering problems may involve anisotropic scatterers.
In this case, the uniaxial PML method has a big advantage over the circular or spherical PML method
as it provides greater flexibility and efficiency in solving such problems.
Thus, in this section, we propose and study the uniaxial PML method for solving the time-domain electromagnetic
scattering problem \eqref{maxwell1}-\eqref{smrc}.

\subsection{The PML equation in the Cartesian coordinates}\label{subsec_upml}

In this subsection, we derive the PML equation in the Cartesian coordinates. To this end,
\begin{figure}[!htbp]
\setcounter{subfigure}{0}
  \centering
  \includegraphics[width=3in]{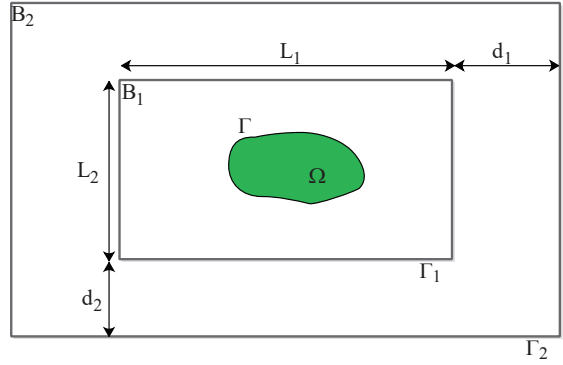}
 \caption{Geometric configuration of the uniaxial PML}\label{geometry}
\end{figure}
define $B_2:=\{x=(x_1,x_2,x_3)^{\top}\in\R^3:|x_j|<L_j/2+d_j,\;j=1,2,3\}$ with boundary $\G_2=\pa B_2$
which is a cubic domain surrounding $B_1$. Denote by $\n_2$ the unit outward normal to $\G_2$.
Let $\Om^{\PML}=B_{2}\ba\ov{B}_1$ be the PML layer and let $\Om_{2}=B_{2}\ba\ov{\Om}$
be the truncated PML domain. See Figure \ref{geometry} for the uniaxial PML geometry.

For $x=(x_1,x_2,x_3)^{\top}\in\R^3$, let $s_1>0$ be an arbitrarily fixed parameter and let us define
the PML medium property as
\ben
\al_j(x_j)=1+s_1^{-1}\sig_j(x_j),\;\;\; j=1,2,3,
\enn
where
\begin{equation}\label{para_upml}
\sig_j(x_j)=\begin{cases}
\ds 0, & |x_j| \leq L_{j}/2,\\
\ds \wit{\sig}_j \left(\frac{|x_j|-L_{j}/2}{d_{j}}\right)^{m},
    & L_{j}/2 < |x_j| \leq L_{j}/2 + d_j,\\
\ds \wit{\sig}_j, & L_{j}/2+d_j<|x_j|<\infty
  \end{cases}
\end{equation}
with positive constants $\wit{\sig}_j$, $j=1,2,3$, and integer $m\geq 1$.
In what follows, we will take the real part of the Laplace transform variable $s\in\C_+$ to be
$s_1$, that is, $\Rt(s)=s_1$.

In the rest of this paper, we always make the following assumptions on the thickness of the PML layer
and the parameters $\wit{\sig}_j$, which are reasonable in our model:
\begin{align}\label{thick_upml}
& d_1=d_2=d_3:=d\geq 1,\;\;\;\;\;\;L=\max\{L_1,L_2,L_3\}\leq C_0d, \\
& \wit{\sig}_1=\wit{\sig}_2=\wit{\sig}_3:=\sig_0>0 \label{para}
\end{align}
for a fixed generic constant $C_0$. Under the assumptions \eqref{thick_upml} and \eqref{para} we have
\be\label{int_sig}
\int_{0}^{{L_{j}}/{2}+d_{j}}\sig_{j}(\tau)d\tau=\frac{\sig_0 d}{m+1},\quad j=1,2,3.
\en
We remark that the constant assumption on $d_j$ and $\wit{\sig}_j$ in \eqref{thick_upml}-\eqref{para}
is only to simplify the convergence analysis but not mandatory.
We now introduce the real stretched Cartesian coordinates $\wit{x}=(\wit{x}_1,\wit{x}_2,\wit{x}_3)^{\top}$ with
\be\label{coord_upml}
\wit{x}_j=\int_0^{x_j}\al_j(\tau)d\tau,\;\;\;j=1,2,3.
\en

Noting that the solution of the exterior problem \eqref{s_exterior_u} in $\R^3\ba\ov{B}_1$ can be
derived as the integral representation \cite[Theorem 12.2]{Monk2003}, we can derive the PML
extension under the stretched coordinates $\wit{x}$ by following \cite{wyz-sinum-2020}.
For any $\bm p\in H^{-1/2}(\Dive,\G_1)$ and $\bm q\in H^{-1/2}(\Dive,\G_1)$, define
\be\label{extension_upml}
\mathbb{E}(\bm p,\bm q)(x):=-\wit{\bm\Psi}_{{\rm SL}}(\bm q)(x)-\wit{\bm\Psi}_{{\rm DL}}(\bm p)(x),
\;\;\;x\in\R^3\ba\ov{B}_1,
\en
where the stretched single- and double-layer potentials are defined as
\ben
\wit{\bm\Psi}_{{\rm SL}}(\bm q)=\int_{\G_1}\wit{\mathbb{G}}^{\top}(s,x,y)\bm q(y)d\g(y),\qquad
\wit{\bm\Psi}_{{\rm DL}}(\bm p)=\int_{\G_1}(\curl_y\wit{\mathbb{G}})^{\top}(s,x,y)\bm p(y)d\g(y).
\enn
Here, the stretched dyadic Green's function is given by
\be\label{Green_upml}
\wit{\mathbb{G}}(s,x,y)=\wit{\Phi}_s(x,y)\mathbb{I}+\frac{1}{k^2}\nabla_y\nabla_y\wit{\Phi}_s(x,y),
\;\;\;x\neq y,\;\;\;k=i\sqrt{\vep\mu}s
\en
with the stretched fundamental solution and the complex distance
\be\label{funda_sol_upml}
\wit{\Phi}_s(x,y)=\frac{e^{-\sqrt{\vep\mu}\rho_s(\wit{x},y)}}{4\pi\rho_s(\wit{x},y)s^{-1}},
\;\;\;\;\;\rho_s(\wit{x},y)=s|\wit{x}-y|.
\en
Introduce the stretched curl operator acting on vector $\bm u=(u_1,u_2,u_3)^{\top}$:
\ben
\wit{\curl}\;{\bm u}=\wit{\nabla}\times{\bm u}:=\left(\frac{\pa u_3}{\pa\wit{x}_2}
-\frac{\pa u_2}{\pa\wit{x}_3},\frac{\pa u_1}{\pa\wit{x}_3}-\frac{\pa u_3}{\pa\wit{x}_1},
\frac{\pa u_2}{\pa\wit{x}_1}-\frac{\pa u_1}{\pa\wit{x}_2}\right)^{\top}
=\mathbb{A}\na\times\mathbb{B}{\bm u}
\enn
with the diagonal matrices
\be\label{BA}
\mathbb{A}=\diag\left\{\frac{1}{\al_2\al_3},\frac{1}{\al_1\al_3},\frac{1}{\al_1\al_2}\right\}
\;\;\;\;{\rm and}\;\;\;\;\mathbb{B}=\diag\{\al_1,\al_2,\al_3\}.
\en
The PML extension in the $s$-domain in $\R^3\ba\ov{B}_1$ of $\g_t(\ch{\bm E})|_{\G_1}$
and $\g_t(\curl\;\ch{\bm E})|_{\G_1}$ is then defined as
\be\label{em_upml}
\ch{\wit{\bm E}}(x)=\mathbb{E}(\g_t(\ch{\bm E}),\g_t(\curl\;\ch{\bm E})),\;\;\;x\in\R^3\ba\ov{B}_1.
\en
Define $\ch{\wit{\bm H}}(x):=-(\mu s)^{-1}\wit{\curl}\;\ch{\wit{\bm E}}(x)$ for $x\in\R^3\ba\ov{B}_1.$
Then it is easy to see that $(\ch{\wit{\bm E}},\ch{\wit{\bm H}})$ satisfies the Maxwell equation
in the $s$-domain:
\be\label{stretched_maxwell}
\wit{\na}\ti\ch{\wit{\bm{E}}}+\mu s\ch{\wit{\bm{H}}}=\0,\;\;\;\;
\wit{\na}\ti\ch{\wit{\bm{H}}}-\vep s\ch{\wit{\bm{E}}}=\0\;\;\;\gin\;\;\;\R^3\se\ov{B}_1.
\en
Define
\ben
(\bm E^{{\rm PML}},\bm H^{{\rm PML}}):=\mathbb{B}(\mathscr{L}^{-1}(\ch{\wit{\bm{E}}}),
\mathscr{L}^{-1}(\ch{\wit{\bm{H}}})).
\enn
Then $(\bm{E}^{{\rm PML}},\bm{H}^{{\rm PML}})$ can be viewed as the extension in the region $\R^3\ba\ov{B}_1$
of the solution of the problem \eqref{maxwell1}-\eqref{smrc} since, by the fact that $\al_j=1$ on $\G_1$
for $j=1,2,3$ we have $\bm{E}^{{\rm PML}}=\bm E$, $\bm{H}^{{\rm PML}}=\bm H$ on $\G_1$.
If we set $\bm{E}^{{\rm PML}}=\bm E$ and $\bm{H}^{{\rm PML}}=\bm H$ in $\Om_1\times(0,T)$,
then $(\bm E^{{\rm PML}},\bm H^{{\rm PML}})$ satisfies the PML problem:
\begin{eqnarray}\label{upml}
\begin{cases}
\ds\nabla\times\bm{E}^{{\rm PML}}+\mu(\mathbb{BA})^{-1}\pa_t\bm{H}^{{\rm PML}}=\0&\gin\;\;\;
    (\R^3\ba\ov{\Om})\times(0,T),\\
\ds\nabla\times\bm{H}^{{\rm PML}}-\vep(\mathbb{BA})^{-1}\pa_t\bm{E}^{{\rm PML}}=\bm J&\gin\;\;\;
    (\R^3\ba\ov{\Om})\times(0,T),\\
\ds\bm{n}\times\bm E^{{\rm PML}}=\0&\on\;\;\;\G\times(0,T),\\
\ds\bm E^{{\rm PML}}(x,0)=\bm H^{{\rm PML}}(x,0)=\0&\gin\;\;\;\R^3\ba\ov{\Om}.
\end{cases}
\end{eqnarray}
The truncated PML problem in the time domain is to find $(\bm E^{p},\bm H^{p})$,
which is an approximation to $(\bm E,\bm H)$ in $\Om_1$, such that
\begin{eqnarray}\label{tupml}
\begin{cases}
\ds\nabla\times\bm{E}^p+\mu(\mathbb{BA})^{-1}\pa_t\bm{H}^p=\0&\gin\;\;\;\Om_{2}\times(0,T),\\
\ds\nabla\times\bm{H}^p-\vep(\mathbb{BA})^{-1}\pa_t\bm{E}^p=\bm J&\gin\;\;\;\Om_{2}\times(0,T),\\
\ds\bm{n}\times\bm E^p=\0&\on\;\;\;\G\times(0,T),\\
\ds\n_2\times\bm E^p=\0&\on\;\;\;\G_{2}\times(0,T),\\
\ds\bm E^p(x,0)=\bm H^p(x,0)=\0&\gin\;\;\;\Om_{2}.
\end{cases}
\end{eqnarray}

\subsection{Well-posedness of the truncated PML problem}

We now study the well-posedness of the truncated PML problem \eqref{tupml},
employing the Laplace transform technique and a variational method.
Eliminate $\bm H^p$ and take the Laplace transform of \eqref{tupml} to obtain that
\be\label{espml}
\begin{cases}
&\nabla\times\left[(\mu s)^{-1}\mathbb{BA}\nabla\times\ch{\bm E}^{p}\right]
+\vep s(\mathbb{BA})^{-1}\ch{\bm E}^{p}=-\ch{\bm J}\;\;\;\gin\;\;\;\Om_{2},\\
& \bm{n}\times\ch{\bm E}^{p}=\0 \quad\on\;\;\;\G,\\
& \n_2\times\ch{\bm E}^{p}=\0\quad\on\;\;\;\G_{2}.
\end{cases}
\en
The variational formulation of \eqref{espml} can be derived as follows:
find a solution $\ch{\bm E}^p\in H_{0}(\curl,\Om_2)$ such that
\be\label{varform_upml}
a_p\left(\ch{\bm{E}}^p,\bm{V}\right)=-\int_{\Om_1}\ch{\bm J}\cdot\ov{\bm{V}} dx,
\qquad\forall\;\bm V\in H_{0}(\curl,\Om_2),
\en
where the sesquilinear form $a_p(\cdot,\cdot)$ is defined as
\be\label{variation_upml}
a_p\left(\ch{\bm{E}}^p,\bm{V}\right)=\int_{\Om_2}\left[(\mu s)^{-1}\mathbb{BA}(\na\times\ch{\bm{E}}^p)
\cdot(\na\times\ov{\bm V})dx+\vep s(\mathbb{BA})^{-1}\ch{\bm{E}}^p\cdot\ov{\bm V}\right]dx.
\en
We have the following result on the well-posedness of the variational problem \eqref{varform_upml}.

\begin{lemma}\label{lem_ep}
For each $s\in\C_+$ with $\Rt(s)=s_1>0$ the variational problem \eqref{varform_upml} has a unique
solution $\ch{\bm E}^p\in H_0(\curl,\Om_{2})$. Further, it holds that
\be\label{ep_estimate}
\|\nabla\times\ch{\bm E}^p\|_{L^2(\Om_{2})^3}+\|s\ch{\bm E}^p\|_{L^2(\Om_{2})^3}
\les s_1^{-1}(1+s_1^{-1}\sig_0)^2\|s\ch{\bm J}\|_{L^2(\Om_1)^3}.
\en
\end{lemma}

\begin{proof}
By the definition of the diagonal matrix $\mathbb{BA}$ (see \eqref{BA}) and a direct calculation
it easily follows that
\begin{align}\label{estimate_BA}
(1+s_1^{-1}\sig_0)^{-2}&\le|\mathbb{BA}|\le(1+s_1^{-1}\sig_0)\quad\gin\;\;\Om^{\PML},\\ \label{estimate_BA1}
(1+s_1^{-1}\sig_0)^{-1}&\le|(\mathbb{BA})^{-1}|\le(1+s_1^{-1}\sig_0)^2,\quad\gin\;\;\Om^{\PML}.
\end{align}
Thus, it is derived that
\be\label{coer_ap}
\Rt[a_p(\ch{\bm{E}}^p,\ch{\bm{E}}^p)]\gtrsim\frac{1}{(1+s_1^{-1}\sig_0)^2}\frac{s_1}{|s|^2}
\left(\|\na\times\ch{\bm E}^p\|_{L^2(\Om_2)^3}+\|s\ch{\bm E}^p\|_{L^2(\Om_2)^3}\right).
\en
The existence and uniqueness of solutions to the problem \eqref{varform_upml} then follow
from the Lax-Milgram theorem.
The estimate \eqref{ep_estimate} can be obtained by combining \eqref{varform_upml}, \eqref{coer_ap}
and the Cauchy-Schwartz inequality. The proof is thus complete.
\end{proof}

To show the well-posedness of the truncated PML problem \eqref{tupml} in the time domain, we need
the following lemma which is the analog of the Paley-Wiener-Schwartz theorem for the Fourier transform
of the distributions with compact support in the case of Laplace transform \cite[Theorem 43.1]{treves1975}.

\begin{lemma}{\rm \cite[Theorem 43.1]{treves1975}.}\label{treves}
Let $\ch{\bm{\om}}(s)$ denote a holomorphic function in the half complex plane $s_1=\Rt(s)>\sig_0$
for some $\sig_0\in\R$, valued in the Banach space $\E$. Then the following statements are equivalent:
\begin{enumerate}[1)]
\item there is a distribution $\om\in\mathcal{D}_+^{'}(\E)$ whose Laplace transform is equal to
$\ch{\bm{\om}}(s)$, where $\mathcal{D}_+^{'}(\E)$ is the space of distributions on the real line
which vanish identically in the open negative half-line;
\item there is a $\sig_1$ with $\sig_0\leq\sig_1<\infty$ and an integer $m\geq0$ such that for
all complex numbers $s$ with $s_1=\Rt(s)>\sig_1$ it holds that $\V\ch{\bm{\om}}(s)\V_{\E}\les(1+|s|)^m$.
\end{enumerate}
\end{lemma}

The well-posedness and stability of the truncated PML problem \eqref{tupml} can be proved
by using Lemmas \ref{lem_ep} and \ref{treves} and the energy method (cf. \cite[Theorem 3.1]{chen2008}).

\begin{theorem}\label{thm_well_tupml}
Let $s_1=1/T$. Then the truncated PML problem \eqref{tupml} in the time domain has a unique solution
$(\bm E^p(x,t),\bm H^p(x,t))$ with
\ben
&\bm E^p\in L^2\big(0,T;H_0(\curl,\Om_{2})\big)\cap H^1\left(0,T;L^2(\Om_{2})^3\right),\\
&\bm H^p\in L^2\big(0,T;H_0(\curl,\Om_{2})\big)\cap H^1\left(0,T;L^2(\Om_{2})^3\right)
\enn
and satisfying the stability estimate
\begin{align}
&\max\limits_{t\in[0,T]}\left[\|\pa_t\bm E^p\|_{L^2(\Om_{2})^3}+\|\nabla\times\bm E^p\|_{L^2(\Om_{2})^3}
  +\|\pa_t\bm H^p\|_{L^2(\Om_{2})^3}+\|\nabla\times\bm H^p\|_{L^2(\Om_{2})^3}\right]\no \\ \label{stability_tupml}
&\qquad\les(1+\sig_0T)^3\|\bm J\|_{H^1(0,T;L^2(\Om_1)^3)}.
\end{align}
\end{theorem}

To study the convergence of the uniaxial PML method, we introduce the EtM operator
$\wih{\mathscr{B}}:  H^{-1/2}(\Curl,\G_1)\to H^{-1/2}(\Dive,\G_1)$ associated with
the truncated PML problem \eqref{espml} in the $s$-domain. Given $\bm\la \in H^{-1/2}(\Dive,\G_1)$,
define
\be\label{etm_pml}
\wih{\mathscr{B}}(\bm\la\times\n_1):=\n_1\times(\mu s)^{-1}\na\times{\bm u}\qquad\on \;\;\G_1,
\en
where $\bm u$ satisfies the following problem in the PML layer:
\be\label{eq_pml}
\begin{cases}
&\na\times\left[(\mu s)^{-1}\mathbb{BA}\na\times{\bm u}\right]+\vep s(\mathbb{BA})^{-1}{\bm u}=0
\qquad\gin\;\;\Om^{\PML}, \\
&\n_1\times{\bm u}=\bm\la\qquad\on\;\;\G_1,\qquad\n_2\times{\bm u}=\0\qquad\on\;\;\G_2.
\end{cases}
\en
We need to show that \eqref{eq_pml} has unique solution, so $\wih{\mathscr{B}}$ is well-defined.
To this end, we consider the following general problem with the tangential trace $\bm\xi$
on $\G_2$, which is needed for the convergence analysis of the PML method:
\be\label{eq_pml_1}
\begin{cases}
&\na\times\left[(\mu s)^{-1}\mathbb{BA}\na\times{\bm u}\right]
 +\vep s(\mathbb{BA})^{-1}{\bm u}=0\qquad\gin\;\;\Om^{\PML},\\
&\n_1\times{\bm u}=\bm\la\qquad\on\;\;\G_1,\qquad\n_2\times{\bm u}=\bm\xi\qquad\on\;\;\G_2.
\end{cases}
\en
Define the sesquilinear form $a^{\rm PML}:H(\curl,\Om^{\PML})\times H(\curl,\Om^{\PML})\ra\C$ as
\be\label{apml}
a^{\PML}(\bm u,\bm V):=\int_{\Om^{\rm PML}}(\mu s)^{-1}\mathbb{BA}(\na\times\bm u)
\cdot(\na\times\ov{\bm{V}})dx+\int_{\Om^{\PML}}\vep s(\mathbb{BA})^{-1}\bm u\cdot\ov{\bm{V}}dx.
\en
Then the variational formulation of \eqref{eq_pml_1} is as follows: Given $\bm\la\in H^{-1/2}(\Dive,\G_1)$
and $\bm\xi\in H^{-1/2}(\Dive,\G_2)$, find $\bm u\in H(\curl,\Om^{\PML})$ such that
$\n_1\times\bm u=\bm\la\;\on\;\G_1$, $\n_2\times\bm u=\bm\xi\;\on\;\G_2$ and
\be\label{apml_form}
a^{\PML}(\bm u,\bm V)=0,\;\;\;\forall\;\bm V\in H_0(\curl,\Om^{\PML}).
\en
Arguing similarly as in proving \eqref{coer_ap}, we obtain that for any $\bm V\in H_0(\curl,\Om^{\rm PML})$,
\be\label{coer_apml}
\Rt\left[a^{\PML}(\bm V,\bm V\right]\gtrsim\frac{1}{(1+s_1^{-1}\sig_0)^2}\frac{s_1}{|s|^2}
\left[\|\nabla\times\bm V\|_{L^2(\Om^{\rm PML})^3}^2+\|s\bm V\|_{L^2(\Om^{\rm PML})^3}^2\right].
\en
By \eqref{coer_apml} and the Lax-Milgram theorem it follows that the variational problem
\eqref{apml_form} has a unique solution. We have the following stability result for the solution
to the problem \eqref{eq_pml_1}.

\begin{lemma}\label{lem_stability_pml}
For any $\bm\la\in H^{-1/2}(\Dive,\G_1)$ and $\bm\xi\in H^{-1/2}(\Dive,\G_2)$, let $\bm u$ be
the solution to the problem \eqref{eq_pml_1}. Then
\be\label{u_estimate}
&\|\na\times\bm u\|_{L^2(\Om^{\PML})^3}+\|s\bm u\|_{L^2(\Om^{\PML})^3} \\
&\qquad\quad\les {s_1^{-1}}{(1+s_1^{-1}\sig_0)^4}|s|(1+|s|)
 (\|\bm\la\|_{H^{-1/2}(\Dive,\G_1)}+\|\bm\xi\|_{H^{-1/2}(\Dive,\G_2)}).
\en
\end{lemma}

\begin{proof}
Let $\bm u_0\in H(\curl,\Om^{\PML})$ be such that $\n_1\times\bm u_0=\bm\la$,
$\n_2\times\bm u_0=\bm\xi$ on $\G_{2}$.
Then, by \eqref{apml_form} we have $\bm\om:=\bm u-\bm u_0\in H_0(\curl,\Om^{\PML})$ and
\be\label{minus_apml}
a^{\PML}(\bm\om,\bm V)=-a^{\PML}(\bm u_0,\bm V),\;\;\;\forall\;\bm V\in H_0(\curl,\Om^{\rm PML}).
\en
This, combined with \eqref{apml}-\eqref{coer_apml} and the Cauchy-Schwartz inequality, gives
\ben
&\frac{1}{(1+s_1^{-1}\sig_0)^2}\frac{s_1}{|s|^2}\left(\|\nabla\times\bm\om\|_{L^2(\Om^{\rm PML})^3}^2
  +\|s\bm\om\|_{L^2(\Om^{\rm PML})^3}^2\right)\\
&\quad\les\Rt\left[a^{\rm PML}(\bm\om,\bm\om)\right]\\
&\quad\les\frac{(1+s_1^{-1}\sig_0)^2}{|s|}\sqrt{1+|s|^2}\left(\|\nabla\times\bm\om\|_{L^2(\Om^{\rm PML})^3}^2
  +\|s\bm\om\|_{L^2(\Om^{\rm PML})^3}^2\right)^{1/2}\|\bm u_0\|_{H(\curl,\Om^{\PML})},
\enn
yielding
\ben
\left(\|\nabla\times\bm\om\|_{L^2(\Om^{\rm PML})^3}^2+\|s\bm\om\|_{L^2(\Om^{\rm PML})^3}^2\right)^{1/2}
\les\frac{(1+s_1^{-1}\sig_0)^4|s|\sqrt{1+|s|^2}}{s_1}\|\bm u_0\|^2_{H(\curl,\Om^{\PML})}.
\enn
This, together with the definition of $\bm\om$ and the Cauchy-Schwartz inequality, implies that
\ben
\|\nabla\times\ch{\bm u}\|_{L^2(\Om^{\rm PML})^3}+\|s\ch{\bm u}\|_{L^2(\Om^{\rm PML})^3}
\les\frac{(1+s_1^{-1}\sigma_0)^4|s|(1+|s|)}{s_1}\|\bm u_0\|_{H(\curl,\Om^{\PML})}.
\enn
The desired estimate \eqref{u_estimate} then follows from the trace theorem.
\end{proof}

Now, by using $\wih{\mathscr{B}}$ the truncated PML problem \eqref{espml} for the electric field
$\ch{\bm E}^{p}$ can be equivalently reduced to the boundary value problem in $\Om_1$:
\be\label{reduced_pml}
\begin{cases}
&\na\times\left[(\mu s)^{-1}\na\times\ch{\bm E}^{p}\right]+\vep s\ch{\bm E}^{p}
  =-\ch{\bm J}\;\;\gin\;\;\Om_{1},\\
&\n\times\ch{\bm E}^{p}=\0 \quad\on\;\;\G,\qquad
\wih{\mathscr{B}}[\ch{\bm E}^{p}_{\G_1}]=\n_1\times(\mu s)^{-1}\na\times\ch{\bm E}^{p}\quad\on\;\;\G_1.
\end{cases}
\en
Similarly, for the problem \eqref{reduced_pml} we can derive its equivalent variational formulation:
find $\ch{\bm E}^p\in H_{\G_1}(\curl,\Om_1)$ such that
\be\label{ahat}
\wih{a}\left(\ch{\bm{E}}^p,\bm{V})\right)=-\int_{\Om_1}\ch{\bm J}\cdot\ov{\bm{V}}dx,
\quad\forall\;\bm V\in H_{\G_1}(\curl,\Om_1),
\en
where the sesquilinear form $\wih{a}(\cdot,\cdot)$ is defined as
\be\label{ahat_form}
\wih{a}\left(\ch{\bm{E}}^p,\bm{V}\right)
:=\int_{\Om_1}\left[(\mu s)^{-1}(\na\times\ch{\bm{E}}^p)\cdot(\na\times\ov{\bm V})dx
+\vep s\ch{\bm{E}}^p\cdot\ov{\bm V}\right]dx
+\langle\wih{\mathscr{B}}[\ch{\bm E}_{\G_1}],\bm V_{\G_1}\rangle_{\G_1}.
\en
By using $\wih{\mathscr{B}}$ and the Laplace and inverse Laplace transform,
the truncated PML problem \eqref{tupml} is equivalent to the initial boundary value problem in $\Om_1$:
\begin{eqnarray}\label{tupml_reduced}
\begin{cases}
\ds\nabla\times\bm{E}^p+\mu\pa_t\bm{H}^p=\0&\gin\;\;\Om_{1}\times(0,T),\\
\ds\nabla\times\bm{H}^p-\vep\pa_t\bm{E}^p=\bm J&\gin\;\;\Om_{1}\times(0,T),\\
\ds\bm{n}\times\bm E^p=\0&\on\;\;\G\times(0,T),\\
\ds\bm E^p(x,0)=\bm H^p(x,0) = \0 & \gin\;\;\Om_1,\\
\ds\hat{\mathscr{T}}[\bm E^p_{\G_1}]=\bm H^p\times\n_1 &\on\;\;\G_1\times(0,T).
\end{cases}
\end{eqnarray}
Here, $\hat{\mathscr{T}}=\mathscr{L}^{-1}\circ\widehat{\mathscr{B}}\circ\mathscr{L}$ is the time-domain EtM operator
for the PML problem.

\subsection{Exponential convergence of the uniaxial PML method}

In this subsection, we prove the exponential convergence of the uniaxial PML method.
We begin with the following lemma which is useful in the proof of the exponential decay property
of the stretched fundamental solution $\wit{\Phi}_s(x,y)$.

\begin{lemma}\label{lem_d_upml}
Let $s=s_1+is_2$ with $s_1>0,\;s_2\in\R$. Then, for any $x\in\G_{2}$ and $y\in\G_1$
the complex distance $\rho_s$ defined by \eqref{funda_sol_upml} satisfies
\ben
|\rho_s(\wit{x},y)/s|\ge d,\;\;\;\;\Rt[\rho_s(\wit{x},y)]\ge\frac{\sig_0 d}{m+1}.
\enn
\end{lemma}

\begin{proof}
For $x\in\G_{2}$ and $y\in\G_1$, $\wit{x}_j-y_j=(x_j-y_j)+s_1^{-1}x_j\hat{\sig}_j(x_j)$, where
\ben
\hat{\sig}_j(x_j)=\frac{1}{x_j}\int_0^{x_j}\sig_j(\tau) d\tau.
\enn
Then, by the definition of the complex distance $\rho_s(\wit{x},y)$ (see \eqref{funda_sol_upml}) we have
\ben
|\rho_s(\wit{x},y)/s|
& =|\wit{x}-y|=\sqrt{(\wit{x}_1-y_1)^2 + (\wit{x}_2-y_2)^2+(\wit{x}_3-y_3)^2} \\
& =\left(\sum_{j=1}^{3}\left[(x_j-y_j)^2+2s_1^{-1}x_j\hat{\sig}_j(x_j)(x_j-y_j)
  +s_1^{-2}x_j^2\hat{\sig}_j^2(x_j)\right]\right)^{1/2} \\
& \ge|x-y|\ge d,
\enn
where we have used the fact that $x_j\hat{\sig}_j(x_j)(x_j-y_j)\ge 0$ for $x\in\G_{2}$ and $y\in\G_1$.
In addition,
\ben
\Rt\left[\rho_s(\wit{x},y)\right]& =\Rt\left[s^2\left((\wit{x}_1-y_1)^2+(\wit{x}_2-y_2)^2
    +(\wit{x}_3-y_3)^2\right)\right]^{1/2}\\
& =s_1\sqrt{(\wit{x}_1-y_1)^2+(\wit{x}_2-y_2)^2+(\wit{x}_3-y_3)^2}\\
& \ge\left(\sum_{j=1}^{3} x_j^2\hat{\sig}_j^2(x_j)\right)^{1/2}.
\enn
If $x_j=\pm(L_j/2+d_j)\in\G_2$, then, by \eqref{int_sig} we have $|x_j\hat{\sig}_j(x_j)|=\sig_0 d/(m+1)$.
Thus, $\Rt\left[\rho_s(\wit{x},y)\right] \ge \sig_0 d/(m+1)$. The proof is thus complete.
\end{proof}

By Lemma \ref{lem_d_upml}, and arguing similarly as in the proof of \cite[Lemma 5.3]{wyz-sinum-2020},
we have similar estimates as in \cite[Lemma 5.3]{wyz-sinum-2020}
for the stretched dyadic Green's function $\wit{\mathbb{G}}$ in the PML layer.
Then we have the decay property of the PML extension (cf. \cite[Theorem 5.4]{wyz-sinum-2020}).

\begin{lemma}\label{lem_extension}
For any $\bm p,\;\bm q\in H^{-1/2}(\Dive,\G_1)$ let $\mathbb{E}(\bm p,\bm q)$ be the PML extension
in the $s$-domain defined in \eqref{extension_upml}. Then we have that for any $x\in\Om^{\PML}$,
\begin{eqnarray}\label{4.2d}
&&|\mathbb{E}(\bm p,\bm q)(x)|\\ \no
&&\les s_1^{-2}d^{1/2}(1+s_1^{-1}\sigma_0)^2e^{-\frac{\sqrt{\vep\mu}\sig_0d}{m+1}}
\left[(1+|s|)\|\bm q\|_{H^{-1/2}(\Dive,\G_1)}+(1+|s|^2)\|\bm p\|_{H^{-1/2}(\Dive,\G_1)}\right]
\end{eqnarray}
and
\begin{eqnarray}\label{4.2e}
&&|\curl_{\wit{x}}\;\mathbb{E}(\bm p,\bm q)(x)|\\ \no
&&\les d^{1/2}(1+s_1^{-1}\sigma_0)^3e^{-\frac{\sqrt{\vep\mu} \sig_0d}{m+1}}
\left[(1+|s|^2)\|\bm q\|_{H^{-1/2}(\Dive,\G_1)}+(1+|s|^3)\|\bm p\|_{H^{-1/2}(\Dive,\G_1)}\right].
\end{eqnarray}
\end{lemma}

We now establish the $L^{2}$-norm and $L^{\ify}$-norm error estimates in time between solutions to
the original scattering problem and the truncated PML problem (\ref{tupml}) in the computational domain $\Om_1$.

\begin{theorem}\label{thm_convergence_upml}
Let $(\bm E,\bm H)$ and $(\bm E^p,\bm H^p)$ be the solutions of the problems \eqref{maxwell1}-\eqref{smrc}
and $(\ref{tupml})$ with $s_1=1/T$, respectively. If the assumptions $(\ref{assump_upml})$ and
$(\ref{assump1_upml})$ are satisfied, then we have the error estimates
\be\label{error_estimate}
&\|\bm E-\bm E^p\|_{L^2(0,T;H(\curl,\Om_1))}+\|\bm H-\bm H^p\|_{L^2(0,T;H(\curl,\Om_1))}\\
&\;\;\les T^{5}d^2(1+\sigma_0T)^{15}e^{-\sig_0d\sqrt{\vep\mu}/2}\|\bm J\|_{H^{10}(0,T;L^2(\Om_1)^3)}.
\en
and
\be\label{Linfinity_estimate}
&\|\bm E-\bm E^p\|_{L^\ify(0,T;H(\curl,\Om_1))}+\|\bm H-\bm H^p\|_{L^\ify(0,T;H(\curl,\Om_1))}\\
&\;\;\les T^{11/2}d^2(1+\sig_0T)^{15}e^{-\sig_0d\sqrt{\vep\mu}/2}\|\bm J\|_{H^{9}(0,T;L^2(\Om_1)^3)}.
\en
\end{theorem}

\begin{proof}
We first prove \eqref{error_estimate}. Let $\bm U=\bm E-\bm E^p$ and $\bm V=\bm H-\bm H^p$ and let $\ch{\bm E}$
and $\ch{\bm E}^p$ be the solutions to the variational problems \eqref{a} and \eqref{ahat}, respectively.
Then, by \eqref{a} and \eqref{ahat} we get
\be\label{auu}
a(\ch{\bm U},\ch{\bm U})=\wih{a}(\ch{\bm E},\ch{\bm U})-a(\ch{\bm E}^p,\ch{\bm U})
=\langle(\wih{\mathscr{B}}-\mathscr{B})[\ch{\bm E}^p_{\G_1}],\ch{\bm U}_{\G_1}\rangle_{\G_1}.
\en
This, together with the uniform coercivity \eqref{coer_a} of $a(\cdot,\cdot)$, implies that
\be\label{u_norm}
\|\ch{\bm U}\|_{H(\curl,\Om_1)}\les s_1^{-1}\left(1+|s|^{2}\right)
\|(\wih{\mathscr{B}}-\mathscr{B})[\ch{\bm E}^p_{\G_1}]\|_{H^{-1/2}(\Dive,\G_1)}.
\en
From the Maxwell equations in $\Om_1$ obtained by taking the Laplace transform of the
problems \eqref{maxwell1}-\eqref{smrc} and $(\ref{tupml})$, it follows that
\ben
\|\ch{\bm V}\|_{H(\curl,\Om_1)}\les(|s|+|s|^{-1})\|\ch{\bm U}\|_{H(\curl,\Om_1)}.
\enn
This, combined with \eqref{u_norm}, leads to the result
\begin{align}
&\|\ch{\bm U}\|_{H(\curl,\Om_1)}+\|\ch{\bm V}\|_{H(\curl,\Om_1)}\no\\ \label{em_error}
&\qquad\les s_1^{-1}(|s|^{-1}+|s|^3)\|(\wih{\mathscr{B}}-\mathscr{B})[\ch{\bm E}^p_{\G_1}]\|_{H^{-1/2}(\Dive,\G_1)}.
\end{align}
We now estimate the norm $\|(\wih{\mathscr{B}}-\mathscr{B})[\ch{\bm E}^p_{\G_1}]\|_{H^{-1/2}(\Dive,\G_1)}$.
For $\ch{\bm E}^p|_{\G_1}$ define its PML extension $\ch{\wit{\bm E}^p}$ in the $s$-domain to be
the solution of the exterior problem
\begin{equation*}
\begin{cases}
\ds\wit{\na}\times[(\mu s)^{-1}\wit{\na}\times\bm v]+\vep s\bm v=\0 & \gin\;\;\;\R^3\ba\ov{B}_1,\\
\ds\n_1\times\bm v=\n_1\times\ch{\bm E}^p & \on\;\;\;\G_{1},\\
\ds\hat{x}\times(\mu s\bm{v}\times\hat{x})-\hat{x}\times(\wit{\na}\times\bm{v})=o\left(|\wit{x}|^{-1}\right)
&{\rm as}\;\;\;|\wit{x}|\ra\infty.
\end{cases}
\end{equation*}
By \cite[Theorem 12.2]{Monk2003} it is easy to see that $\ch{\wit{\bm E}^p}$ has the integral representation
\ben
\ch{\wit{\bm E}^p}=\mathbb{E}(\g_t(\ch{\bm E}^p),\g_t(\wit{\curl}\;\ch{\wit{\bm E}^p})).
\enn
Define $\ds\ch{\wit{\bm H}^p}:=-(\mu s)^{-1}\wit{\curl}\ch{\wit{\bm E}^p}$.
Then $(\ch{\wit{\bm E}^p},\ch{\wit{\bm H}^p})$ satisfies the stretched Maxwell equations
in \eqref{stretched_maxwell} in $\R^3\ba\ov{B}_{1}$.
It is worth noting that $\ch{\wit{\bm H}^p}$ is not the extension of $\ch{\bm H^p}|_{\G_1}$.

Noting that $\wit{\na}\times\bm v=\mathbb{A}\na\times\mathbb{B}\bm v$, we know that
$\mathbb{B}\ch{\wit{\bm E}^p}$ satisfies the problem
\begin{equation*}
\begin{cases}
\ds\nabla\times\left[(\mu s)^{-1}\mathbb{BA}\nabla\times\bm v\right]+\vep s(\mathbb{BA})^{-1}\bm v=\0
& \gin\;\;\;\R^3\ba\ov{B}_1,\\
\ds\n_1\times\bm v=\n_1\times\ch{\bm E}^p & \on\;\;\;\G_1,\\
\ds\hat{x}\times(\mu s\mathbb{B}^{-1}\bm{v}\times\hat{x})-\hat{x}\times(\mathbb{A}\na\times\bm v)
=o\left(|\wit{x}|^{-1}\right) &{\rm as}\;\;\;|\wit{x}|\ra\infty,
\end{cases}
\end{equation*}
where we have used the fact that $\ch{\wit{\bm E}^p}$ is the extension of $\ch{\bm E}^p|_{\G_1}$
and $\mathbb{B}=\diag\{1,1,1\}$ on $\G_1$. By the definition of $\mathscr{B}$, and since
$\mathbb{A}=\diag\{1,1,1\}$ on $\G_1$, it is easy to see that
\ben
\mathscr{B}[\ch{\bm E}^p_{\G_1}]=\n_1\times(\mu s)^{-1}\wit{\na}\times\ch{\wit{\bm E}^p}
=\n_1\times(\mu s)^{-1}\na\times\mathbb{B}\ch{\wit{\bm E}^p}.
\enn
By the definition of $\wih{\mathscr{B}}$ in \eqref{etm_pml}, we obtain that
\be\label{error_beta}
(\wih{\mathscr{B}}-\mathscr{B})[\ch{\bm E}^p_{\G_1}]=\n_1\times(\mu s)^{-1}\na\times\bm\om
\en
where $\bm\om$ satisfies
\ben
&\na\times\left[(\mu s)^{-1}\mathbb{BA}\na\times{\bm\om}\right]+\vep s(\mathbb{BA})^{-1}{\bm\om}
=0\quad\gin\;\; \Om^{\PML}, \\
&\n_1\times{\bm\om}=\0\quad\on\;\;\G_1,\qquad
\n_2\times{\bm\om}=\g_t(\mathbb{B}\ch{\wit{\bm E}^p})\quad\on\;\;\G_{2}.
\enn
By Lemma \ref{lem_stability_pml} and the estimate for $\mathbb{BA}$ and $(\mathbb{BA})^{-1}$
in \eqref{estimate_BA}-\eqref{estimate_BA1}, we have
\begin{align}
&\|\n_1\times(\mu s)^{-1}\na\times\bm\om\|_{H^{-1/2}(\Dive,\G_1)}\no \\
&\les (1+s_1^{-1}\sig_0)^2\|(\mu s)^{-1}\mathbb{BA}\na\times\bm\om\|_{H(\curl,\Om^{\PML})}\no\\
&\les (1+s_1^{-1}\sig_0)^2\left(\frac{(1+s_1^{-1}\sig_0)^2}{|s|^2}\|\na\times\bm\om\|^2_{L^2(\Om^{\PML})^3}
     +(1+s_1^{-1}\sig_0)^4\|s\bm\om\|^2_{L^2(\Om^{\PML})^3}\right)^{1/2}\no \\ \label{curl_om}
&\les s_1^{-1}(1+s_1^{-1}\sig_0)^8(1+|s|)^2\|\g_t(\mathbb{B}\ch{\wit{\bm E}^p})\|_{H^{-1/2}(\Dive,\G_2)}.
\end{align}
Since $\wit{\na}\times\bm v=\mathbb{A}\na\times\mathbb{B}\bm v$ and $|\mathbb{A}^{-1}|\le(1+\sig_0)^2$
in $\Om^{\PML}$, we have by the boundedness of the trace operator $\g_t$ that
\be\label{gtbep}
\|\g_t(\mathbb{B}\ch{\wit{\bm E}^p})\|_{H^{-1/2}(\Dive,\G_2)}
\les\|\mathbb{B}\ch{\wit{\bm E}^p}\|_{H(\curl,\Om^{\PML})}
\les (1+s_1^{-1}\sig_0)^2\|\ch{\wit{\bm E}^p}\|_{H(\wit{\curl},\Om^{\PML})}.
\en
By Lemma \ref{lem_extension} and the boundedness of $\g_T$ and $\g_t$ it is derived that
\begin{align}
\|\ch{\wit{\bm E}^p}\|^2_{H(\wit{\curl},\Om^{\PML})}&\le(\|\ch{\wit{\bm E}^p}\|^2_{L^{\infty}(\Om^{\PML})}
+\|\wit{\curl}\ch{\wit{\bm E}^p}\|^2_{L^{\infty}(\Om^{\PML})})|\Om^{\PML}|\no\\
&\les s_1^{-4}d^4(1+s_1^{-1}\sig_0)^6 e^{-2\frac{\sqrt{\vep\mu}\sig_0d}{m+1}}
\Big[(1+|s|^4)\|\g_t(\wit{\curl}\,\ch{\wit{\bm E}^p})\|^2_{H^{-1/2}(\Dive,\G_1)} \no\\
&\;\;\;+(1+|s|^6)\|\g_t\ch{\bm E}^p\|^2_{H^{-1/2}(\Dive,\G_1)}\Big] \no\\
&\les s_1^{-4}d^4(1+s_1^{-1}\sigma_0)^{6}e^{-2\frac{\sqrt{\vep\mu}\sig_0d}{m+1}}
\Big[(1+|s|^4)^2\|\g_T\ch{\wit{\bm E}^p}\|^2_{H^{-1/2}(\Curl,\G_1)} \no\\
&\qquad\qquad\;\;\;+(1+|s|^6)\|\g_t\ch{\bm E}^p\|^2_{H^{-1/2}(\Dive,\G_1)}\Big] \no\\
&\les s_1^{-4}d^4(1+s_1^{-1}\sigma_0)^{6}e^{-2\frac{\sqrt{\vep\mu}\sig_0d}{m+1}}
\sum_{l=0}^{4}\|s^l\ch{\bm E}^p\|^2_{H(\curl,\;\Om_1)} \no\\ \label{ep_curl}
&\les s_1^{-6}d^4(1+s_1^{-1}\sigma_0)^{10}e^{-2\frac{\sqrt{\vep\mu}\sig_0d}{m+1}}
\sum_{l=0}^5\|s^{l}\ch{\bm J}\|^2_{L^2(\Om_1)^3},
\end{align}
where we have used Lemma \ref{lem_ep} and the upper bound estimate (\ref{bound_etm}) of the EtM
operator $\mathscr{B}$. Combining \eqref{em_error}-\eqref{ep_curl} yields that
\begin{align}
&\|\ch{\bm U}\|^2_{H(\curl,\Om_1)}+\|\ch{\bm V}\|^2_{H(\curl,\Om_1)}\no\\
&\;\;\;\les s_1^{-10}d^4(1+s_1^{-1}\sigma_0)^{30}e^{-2\frac{\sqrt{\vep\mu}\sig_0d}{m+1}}
\sum_{l=0}^{10}\|s^{l}\ch{\bm J}\|^2_{L^2(\Om_1)^3}.
\end{align}
This, together with the Parseval identity for the Laplace transform (see \cite[(2.46)]{Cohen2007})
\ben
\frac{1}{2\pi}\int_{-\infty}^{\infty}\ch{\bm{u}}(s)\cdot\ch{\bm{v}}(s)ds_2
=\int_0^{\infty}e^{-2s_1t}\bm{u}(t)\cdot\bm{v}(t)dt,
\enn
for all $s_1>\la$, where $\la$ is the abscissa of convergence for $\ch{\bm{u}}$ and $\ch{\bm{v}}$,
gives
\begin{align}
&\|\bm U\|^2_{L^2(0,T;H(\curl,\Om_1))}+\|\bm V\|^2_{L^2(0,T;H(\curl,\Om_1))} \no\\
&\;\;=\int_{0}^{T}\left(\|\bm U\|^2_{H(\curl,\Om_1)}+\|\bm V\|^2_{H(\curl,\Om_1)}\right)dt\no\\
&\;\;\le e^{2s_1T}\int_{0}^{\infty}e^{-2s_1t}\left(\|\bm U\|^2_{H(\curl,\Om_1)}
  +\|\bm V\|^2_{H(\curl,\Om_1)}\right)dt\no\\
&\;\;\les e^{2s_1T}\int_{0}^{\infty}s_1^{-10}d^4(1+s_1^{-1}\sigma_0)^{30}e^{-2\frac{\sqrt{\vep\mu}\sig_0d}{m+1}}
  \sum_{l=0}^{10}\|s^{l}\ch{\bm J}\|^2_{L^2(\Om_1)^3} ds_2\no\\ \label{eh_terror}
&\;\;\les e^{2s_1T}s_1^{-10}d^4(1+s_1^{-1}\sigma_0)^{30}e^{-2\frac{\sqrt{\vep\mu}\sig_0d}{m+1}}
  \|\bm J\|^2_{H^{10}(0,T;L^2(\Om_1)^3)},
\end{align}
where we have used the assumptions (\ref{assump_upml}) and (\ref{assump1_upml}) to get the last inequality.
It is obvious that $m$ should be chosen small enough to ensure rapid convergence (thus we need to take $m=1$).
Since $s_1^{-1}=T$ in (\ref{eh_terror}), we obtain the required estimate \eqref{error_estimate} by using
the Cauchy-Schwartz inequality.

We now prove \eqref{Linfinity_estimate}. Since $(\bm E,\bm H)$ and $(\bm E^p,\bm H^p)$ satisfy
the equations \eqref{reduced} and \eqref{tupml_reduced}, respectively, it is easy to verify
that $(\bm U,\bm V)$ satisfies the problem
\begin{equation}\label{eqn_U_V}
  \begin{cases}
    \na\times\bm{U} +\mu\pa_t\bm{V} = \0 & \gin\;\;\;\Om_1\times(0,T),\\
    \na\times\bm{V} -\vep\pa_t\bm{U} = \bm 0 & \gin\;\;\;\Om_1\times(0,T),\\
    \n\times\bm{U} =\0 & \on\;\;\;\G\times(0,T),\\
    \bm U(x,0)=\bm V(x,0) = \0 & \gin\;\;\;\Om_1,\\
    \bm V\ti\n_1=(\mathscr{T}-\hat{\mathscr{T}})[\bm E^p_{\G_1}]+\mathscr{T}[\bm U_{\G_1}] & \on\;\;\;\G_1\times(0,T).
  \end{cases}
\end{equation}
Eliminating $\bm V$ yields that
\begin{equation}\label{eqn_U}
\begin{cases}
\na\times(\mu^{-1}\na\times\bm{U}) +\vep\pa_t^2\bm{U} = \0 & \gin\;\;\;\Om_1\times(0,T),\\
\n\times\bm{U} =\0 & \on\;\;\;\G\times(0,T),\\
\bm U(x,0)=\pa_t\bm U(x,0) = \0 & \gin\;\;\;\Om_1,\\
\mu^{-1}(\na\times\bm{U})\times\n_1+\mathscr{C}[\bm U_{\G_1}]=(\hat{\mathscr{T}}-\mathscr{T})[\pa_t\bm E^p_{\G_1}]
   & \on\;\;\;\G_1\times(0,T),
\end{cases}
\end{equation}
where $\mathscr{C}=\mathscr{L}^{-1}\circ s\mathscr{B}\circ \mathscr{L}$.
The variational problem of \eqref{eqn_U} is to find $\bm U\in H_{\G}(\curl,\Om_1)$ for all $t>0$ such that
\begin{align}\label{variation_td}
\int_{\Om_1}\vep\pa_t^2\bm U\cdot\ov{\bm\om}dx= &-\int_{\Om_1}\mu^{-1}(\na\times\bm{U})(\na\times\ov{\bm{\om}})dx\\ \no
&+\int_{\G_1}(\hat{\mathscr{T}}-\mathscr{T})[\pa_t\bm E^p_{\G_1}]\cdot\ov{\bm{\om}}_{\G_1}d\g
  -\int_{\G_1}\mathscr{C}[\bm U_{\G_1}]\cdot\ov{\bm{\om}}_{\G_1} d\g,\quad\forall\bm{\om}\in H_{\G}(\curl,\Om_1).
\end{align}
For $0<\xi<T$, introduce the auxiliary function
\begin{align*}
{\bm\Psi}_1(x,t)=\int_t^\xi\bm U(x,\tau)d\tau,\quad x\in\Om_1,0\leq t\leq\xi.
\end{align*}
Then it is easy to verify that
\begin{align}\label{3.28}
 {\bm\Psi}_1(x,\xi)=0,\quad\pa_t{\bm\Psi}_1(x,t)=-\bm U(x,t).
\end{align}
For any $\bm\phi(x,t)\in L^2\left(0,\xi;L^2(\Om_1)^3\right)$, using integration by parts and
condition (\ref{3.28}), we have
\begin{align}\label{3.29}
\int_0^\xi\bm\phi(x,t)\cdot\ov{{\bm\Psi}}_1(x,t)dt
 =\int_0^\xi\int_0^t\bm\phi(x,\tau)d\tau\cdot\ov{\bm U}(x,t)dt.
\end{align}
Taking the test function $\bm\om={\bm\Psi}_1$ in \eqref{variation_td} and using (\ref{3.28}) give
\begin{align}\label{3.30}
\Rt\int_0^\xi\int_{\Om_1}\vep\pa_t^2\bm U\cdot\ov{{\bm\Psi}}_1dxdt
=&\Rt\int_{\Om_1}\int_0^\xi\vep\Big(\pa_t(\pa_t\bm U\cdot\ov{{\bm\Psi}}_1)+\pa_t\bm U\cdot\ov{\bm U}\Big)dtdx\no\\
=&\frac{1}{2}\|\sqrt{\vep}\bm U(\cdot,\xi)\|_{L^2(\Om_1)^3}^2.
\end{align}
By (\ref{3.29}) we have the estimate
\begin{align*}
&\Rt\int_0^\xi\int_{\Om_1}\mu^{-1}(\na\times\bm U)\cdot(\na\ti\ov{{\bm\Psi}}_1)dxdt\\
&=\Rt\int_{\Om_1}\int_0^\xi\mu^{-1}(\na\times\bm U)\cdot\int_t^\xi(\na\times\ov{\bm U}(x,\tau))d\tau dt dx\\
&=\int_{\Om_1}\mu^{-1}\Big|\int_0^\xi(\na\times\bm U)(x,t)dt\Big|^2dx
-\Rt\int_0^\xi\int_{\Om_1}\mu^{-1}(\na\times\bm U)\cdot(\na\times\ov{{\bm\Psi}}_1)dxdt,
\end{align*}
which implies that
\begin{align}\label{3.31}
\Rt\int_0^\xi\int_{\Om_1}\mu^{-1}(\na\ti\bm U)\cdot(\na\ti\ov{{\bm\Psi}}_1)dxdt
=\frac{1}{2}\int_{\Om_1}\mu^{-1}\Big|\int_0^\xi\na\ti\bm U(x,t)dt\Big|^2dx.
\end{align}
Integrating \eqref{variation_td} from $t=0$ to $t=\xi$ and taking the real parts yield
\begin{align}\label{3.37}
&\frac{1}{2}\|\sqrt{\vep}\bm U(\cdot,\xi)\|_{L^2(\Om_1)^3}^2
 +\frac{1}{2}\int_{\Om_1}\mu^{-1}\Big|\int_0^\xi\na\ti\bm U(x,t)dt\Big|^2\no\\
&=\Rt\int_0^\xi\int_{\G_1}(\hat{\mathscr{T}}-\mathscr{T})[\pa_t\bm E^p_{\G_1}]\cdot{\ov{{\bm\Psi}}_1}_{\G_1}d\g dt
  -\Rt\int_0^\xi\int_{\G_1}\mathscr{C}[\bm U_{\G_1}]\cdot{\ov{{\bm\Psi}}_1}_{\G_1} d\g dt.
\end{align}
First, using (\ref{3.29}) and Lemma \ref{lem_positive}, we have
\begin{align}\label{3.38}
\Rt\int_0^\xi\int_{\G_1} \mathscr{C}[\bm U_{\G_1}] \cdot {\ov{{\bm\Psi}}_1}_{\G_1} d\g dt
=\Rt\int_{\G_1}\int_0^\xi\left(\int_0^t\mathscr{C}[\bm U_{\G_1}](x,\tau)d\tau\right)\cdot\ov{\bm U}_{\G_1}(x,t)dt\,d\g\ge 0.
\end{align}
Then, and by \eqref{3.29} we deduce the estimate
\be\label{3.39}
&\frac{1}{2}\|\sqrt{\vep}\bm U(\cdot,\xi)\|_{L^2(\Om_1)^3}^2
 +\frac{1}{2}\int_{\Om_1}\mu^{-1}\Big|\int_0^\xi\na\ti\bm U(x,t)dt\Big|^2 dx\\
&\le\Rt\int_0^\xi\int_{\G_1}(\hat{\mathscr{T}}-\mathscr{T})[\pa_t\bm E^p_{\G_1}]\cdot{\ov{{\bm\Psi}}_1}_{\G_1}d\g dt\\
&=\Rt\int_0^\xi\int_{\G_1}\left(\int_0^t(\hat{\mathscr{T}}-\mathscr{T})[\pa_\tau\bm E^p_{\G_1}]d\tau\right)
\ov{\bm U}_{\G_1}(x,t) d\g dt\\
&\les\left(\int_0^\xi\|(\hat{\mathscr{T}}-\mathscr{T})[\pa_t\bm E^p_{\G_1}](\cdot,t)\|_{H^{-1/2}(\Dive,\G_1)}dt\right)
\left(\int_0^\xi\|\bm U(\cdot,t)\|_{H(\curl,\Om_1)} dt\right),
\en
where we have used the trace theorem to get the last inequality.
The right-hand of \eqref{3.39} contains the term
\ben
\int_0^\xi\|\bm U(\cdot,t)\|_{H(\curl,\Om_1)} dt=\int_0^\xi\left(\|\bm U(\cdot,t)\|^2_{L^2(\Om_1)^3}
+\|\na\times\bm U(\cdot,t)\|^2_{L^2(\Om_1)^3}\right)^{\frac{1}{2}}dt
\enn
which cannot be controlled by the left-hand of \eqref{3.39}. To address this issue, we consider the new problem
\begin{equation}\label{eqn_U_t}
\begin{cases}
\na\times(\mu^{-1}\na\times(\pa_t\bm{U})) +\vep\pa_t^2(\pa_t\bm{U})=\0 & \gin\;\;\;\Om_1\times(0,T),\\
\n\times\pa_t\bm{U} =\0 & \on\;\;\;\G\times(0,T),\\
\pa_t\bm{U}(x,0)=\pa_t^2\bm U(x,0) = \0 & \gin\;\;\;\Om_1,\\
\mu^{-1}(\na\times(\pa_t\bm{U}))\times\n_1+\mathscr{C}[\pa_t\bm U_{\G_1}]
=(\hat{\mathscr{T}}-\mathscr{T})[\pa_t^2\bm E^p_{\G_1}] & \on\;\;\;\G_1\times(0,T),
\end{cases}
\end{equation}
which is obtained by differentiating each equation of \eqref{eqn_U} with respect to $t$.
By a similar argument as in deriving \eqref{variation_td}, we obtain the variational formulation
of \eqref{eqn_U_t}: find $u$ such that for all $\bm{\om}\in H_{\G}(\curl,\Om_1)$,
\begin{align}\no
\int_{\Om_1}\vep\pa_t^2(\pa_t\bm U)\cdot\ov{\bm\om}dx
&=-\int_{\Om_1}\mu^{-1}(\na\times(\pa_t\bm{U}))(\na\times\ov{\bm{\om}})dx\\ \label{variation_td1}
&\;\;\;+\int_{\G_1} (\hat{\mathscr{T}}-\mathscr{T})[\pa_t^2\bm E^p_{\G_1}]\cdot\ov{\bm{\om}}_{\G_1}d\g
  -\int_{\G_1}\mathscr{C}[\pa_t\bm U_{\G_1}] \cdot \ov{\bm{\om}}_{\G_1} d\g.
\end{align}
Define the auxiliary function
\begin{align*}
{\bm\Psi}_2(x,t)=\int_t^\xi\pa_{\tau}\bm U(x,\tau)d\tau,\quad x\in\Om_1,0\leq t\leq\xi.
\end{align*}
Similarly as in the derivation of \eqref{3.30}-\eqref{3.31}, we conclude by integration by parts that
\begin{align}\label{3.30+}
&\Rt\int_0^\xi\int_{\Om_1}\vep\pa_t^2(\pa_t\bm U)\cdot\ov{{\bm\Psi}}_2dxdt
=\frac{1}{2}\|\sqrt{\vep}\pa_t\bm U(\cdot,\xi)\|_{L^2(\Om_1)^3}^2,\\ \label{3.30++}
&\Rt\int_0^\xi\int_{\Om_R}\mu_e^{-1}(\na\ti(\pa_t\bm U))\cdot(\na\ti\ov{\bm\Psi}_2)dxdt
=\frac{1}{2}\Big\|\frac{1}{\sqrt{\mu}}\na\ti\bm U(\cdot,\xi)\Big\|_{L^2(\Om_1)^3}^2.
\end{align}
Choosing the test function $\bm\om={\bm\Psi}_2$ in \eqref{variation_td1}, integrating the resulting equation
with respect to $t$ from $t=0$ to $t=\xi$ and taking the real parts yield
\begin{align}\label{3.37+}
&\frac{1}{2}\|\sqrt{\vep}\pa_t\bm U(\cdot,\xi)\|_{L^2(\Om_1)^3}^2
+\frac{1}{2}\Big\|\frac{1}{\sqrt{\mu}}\na\ti\bm U(\cdot,\xi)\Big\|_{L^2(\Om_1)^3}^2\no\\
&=\Rt\int_0^\xi\int_{\G_1}(\hat{\mathscr{T}}-\mathscr{T})[\pa_t^2\bm E^p_{\G_1}]\cdot{\ov{{\bm\Psi}}_2}_{\G_1}d\g dt
- \Rt\int_0^\xi\int_{\G_1} \mathscr{C}[\pa_t\bm U_{\G_1}] \cdot {\ov{{\bm\Psi}}_2}_{\G_1} d\g dt.
\end{align}
Similarly to \eqref{3.38}, it follows from \eqref{3.29} and Lemma \ref{lem_positive1} that
\ben
\Rt\int_0^\xi\int_{\G_1} \mathscr{C}[\pa_t\bm U_{\G_1}] \cdot {\ov{{\bm\Psi}}_2}_{\G_1} d\g dt \ge 0.
\enn
Thus, and by \eqref{3.37+} we have
\begin{align}\label{3.37++}
&\frac{1}{2}\|\sqrt{\vep}\pa_t\bm U(\cdot,\xi)\|_{L^2(\Om_1)^3}^2
+\frac{1}{2}\Big\|\frac{1}{\sqrt{\mu}}\na\ti\bm U(\cdot,\xi)\Big\|_{L^2(\Om_1)^3}^2\no\\
&\le\Rt\int_0^\xi\int_{\G_1} (\hat{\mathscr{T}}-\mathscr{T})[\pa_t^2\bm E^p_{\G_1}]\cdot{\ov{{\bm\Psi}}_2}_{\G_1}d\g dt\no\\
&=\Rt\int_0^\xi\int_{\G_1}\left(\int_0^t(\hat{\mathscr{T}}-\mathscr{T})[\pa_\tau^2\bm E^p_{\G_1}]d\tau\right)
 \pa_t\ov{\bm U}_{\G_1}(x,t) d\g dt\no\\
&=\Rt\int_0^\xi\int_{\G_1}(\hat{\mathscr{T}}-\mathscr{T})[\pa_t^2\bm E^p_{\G_1}]\cdot\ov{\bm U}_{\G_1}(x,\xi)d\g dt
 -\Rt\int_0^\xi\int_{\G_1}(\hat{\mathscr{T}}-\mathscr{T})[\pa_t^2\bm E^p_{\G_1}]\ov{\bm U}_{\G_1}(x,t)d\g dt\no\\
&\le\int_0^\xi \|(\hat{\mathscr{T}}-\mathscr{T})[\pa_t^2\bm E^p_{\G_1}]\|_{H(\Dive,\G_1)}\cdot
\left(\|\bm U(\cdot,\xi)\|_{H(\curl,\Om_1)}+\|\bm U(\cdot,t)\|_{H(\curl,\Om_1)}\right)dt
\end{align}
Combining \eqref{3.39} and \eqref{3.37++} gives
\begin{align}\label{3.51}
&\|\bm U(\cdot,\xi)\|_{L^2(\Om_1)^3}^2+\|\pa_t\bm U(\cdot,\xi)\|_{L^2(\Om_1)^3}^2
  +\|\na\times\bm U(\cdot,\xi)\|_{L^2(\Om_1)^3}^2\no\\
&\les\left(\int_0^\xi\|(\hat{\mathscr{T}}-\mathscr{T})[\pa_t\bm E^p_{\G_1}](\cdot,t)\|_{H^{-1/2}(\Dive,\G_1)}dt\right)
\left(\int_0^\xi\|\bm U(\cdot,t)\|_{H(\curl,\Om_1)} dt\right)\no\\
&\quad+\int_0^\xi \|(\hat{\mathscr{T}}-\mathscr{T})[\pa_t^2\bm E^p_{\G_1}]\|_{H^{-1/2}(\Dive,\G_1)}\cdot
\left(\|\bm U(\cdot,\xi)\|_{H(\curl,\Om_1)} + \|\bm U(\cdot,t)\|_{H(\curl,\Om_1)}\right)dt.
\end{align}
Taking the $L^{\ify}$-norm of both sides of \eqref{3.51} with respect to $\xi$ and using the Young inequality yield
\ben
&\|\bm U\|_{L^{\ify}(0,T;L^2(\Om_1)^3)}^2+\|\pa_t\bm U\|_{L^{\ify}(0,T;L^2(\Om_1)^3)}^2
+ \|\na\times\bm U\|_{L^{\ify}(0,T;L^2(\Om_1)^3)}^2\\
&\les T^2\|(\hat{\mathscr{T}}-\mathscr{T})[\pa_t\bm E^p_{\G_1}]\|^2_{L^1(0,T;H^{-1/2}(\Dive,\G_1))}
+\|(\hat{\mathscr{T}}-\mathscr{T})[\pa_t^2\bm E^p_{\G_1}]\|^2_{L^1(0,T;H^{-1/2}(\Dive,\G_1))},
\enn
which, together with the Cauchy-Schwartz inequality, implies that
\begin{align}\label{U_estimate}
&\|\bm U\|_{L^{\ify}(0,T;L^2(\Om_1)^3)}+\|\pa_t\bm U\|_{L^{\ify}(0,T;L^2(\Om_1)^3)}
+ \|\na\times\bm U\|_{L^{\ify}(0,T;L^2(\Om_1)^3)}\\\no
&\les T^{3/2} \|(\hat{\mathscr{T}}-\mathscr{T})[\pa_t\bm E^p_{\G_1}]\|_{L^2(0,T;H^{-1/2}(\Dive,\G_1))}
+ T^{1/2} \|(\hat{\mathscr{T}}-\mathscr{T})[\pa_t^2\bm E^p_{\G_1}]\|_{L^2(0,T;H^{-1/2}(\Dive,\G_1))}.
\end{align}
We now only need to estimate the right-hand term of \eqref{U_estimate}.
By \eqref{error_beta} and the definition of $\hat{\mathscr{T}}$ (see \eqref{tupml_reduced}) we know that
$(\hat{\mathscr{T}}-\mathscr{T})[\pa_t\bm E^p_{\G_1}]=\n_1\times\mu^{-1}\na\ti\bm v$,
where $\bm v$ satisfies the problem
\begin{equation}\label{PML_system}
  \begin{cases}
    \na\times(\mu^{-1}\mathbb{BA}\na\times\bm{v}) +\vep(\mathbb{BA})^{-1}\pa_t^2 \bm{v} = \0 & \gin\;\;\;\Om^{\PML}\times(0,T),\\
    \n_1\times\bm{v} =\0 & \on\;\;\;\G_1\times(0,T),\\
    \n_2\times\bm{v} =\g_t(\mathbb{B}\widetilde{\bm E}^p) & \on\;\;\;\G_2\times(0,T),\\
    \bm v(x,0)=\pa_t\bm v(x,0) = \0 & \gin\;\;\;\Om^{\PML}.
  \end{cases}
\end{equation}
Thus we deduce that
\ben
\|(\hat{\mathscr{T}}-\mathscr{T})[\pa_t\bm E^p_{\G_1}]\|^2_{L^2(0,T;H^{-1/2}(\Dive,\G_1))}
&= \|\n_1\times\mu^{-1}\na\ti\bm v\|^2_{L^2(0,T;H^{-1/2}(\Dive,\G_1))}\\
&\le e^{2s_1T}\int_0^\ify e^{-2s_1t}\|\mu^{-1}\na\ti\bm v\|^2_{H(\curl,\Om^{\PML})}dt\\
&\les e^{2s_1T} \int_{-\ify}^\ify\|\mu^{-1}\na\ti\ch{\bm v}\|^2_{H(\curl,\Om^{\PML})}ds_2.
\enn
Repeating \eqref{curl_om}-\eqref{ep_curl} yields
\ben
&\|(\hat{\mathscr{T}}-\mathscr{T})[\pa_t\bm E^p_{\G_1}]\|_{L^2(0,T;H^{-1/2}(\Dive,\G_1))}\\
&\quad\les e^{s_1T} s_1^{-4}d^2(1+s_1^{-1}\sig_0)^{15}e^{-\frac{\sqrt{\vep\mu}\sig_0d}{m+1}}
  \left[\int_{-\ify}^{\ify}\sum_{l=0}^8\|s^{l}\ch{\bm J}\|^2_{L^2(\Om_1)^3} ds_2\right]^{1/2}\\
&\quad\les e^{s_1T} s_1^{-4}d^2(1+s_1^{-1}\sig_0)^{15}e^{-\frac{\sqrt{\vep\mu}\sig_0d}{m+1}}
  \|\bm J\|_{H^{8}(0,T;L^2(\Om_1)^3)}.
\enn
Similarly, we have
\ben
\|(\hat{\mathscr{T}}-\mathscr{T})[\pa_t^2\bm E^p_{\G_1}]\|_{L^2(0,T;H^{-1/2}(\Dive,\G_1))}
\les e^{s_1T} s_1^{-4}d^2(1+s_1^{-1}\sig_0)^{15}e^{-\frac{\sqrt{\vep\mu}\sig_0d}{m+1}}
\|\bm J\|_{H^{9}(0,T;L^2(\Om_1)^3)}.
\enn
By \eqref{U_estimate} and the above two estimates it follows on setting $s_{1}=1/T$ and $m=1$ that
\ben
&\|\bm U\|_{L^{\ify}(0,T;L^2(\Om_1)^3)}+\|\pa_t\bm U\|_{L^{\ify}(0,T;L^2(\Om_1)^3)}
+ \|\na\times\bm U\|_{L^{\ify}(0,T;L^2(\Om_1)^3)}\\
&\les T^{11/2}d^2(1+\sig_0T)^{15}e^{-\sqrt{\vep\mu}\sig_0d/2}\|\bm J\|_{H^{9}(0,T;L^2(\Om_1)^3)}.
\enn
From this, the definition of $\bm U$ and Maxwell's system \eqref{eqn_U_V}
the required estimate \eqref{Linfinity_estimate} then follows. The proof is thus complete.
\end{proof}

\begin{remark}\label{re4.8} {\rm
The $L^2$-norm error estimate \eqref{error_estimate} can also be obtained by integrating \eqref{3.51}
with respect to $\xi$ from $0$ to $T$. The idea of using the uniform coercivity of the variational form
in our proof of the $L^2$-norm error estimate \eqref{error_estimate} is also known for the time-harmonic
PML method. This builds a connection between our proposed time-domain PML method with the real coordinate
stretching technique and the time-harmonic PML method in some sense.
}
\end{remark}

\section{Conclusions}\label{sec_conclusion}

In this paper, by using the real coordinate stretching technique we proposed a uniaxial PML method
in the Cartesian coordinates for 3D time-domain electromagnetic scattering problems, which is of advantage
over the spherical one in dealing with scattering problems involving anisotropic scatterers.
The well-posedness and stability estimates of the truncated uniaxial PML problem in the time domain were
established by employing the Laplace transform technique and the energy argument.
The exponential convergence of the uniaxial PML method was also proved in terms of the thickness
and absorbing parameters of the PML layer, based on the error estimate between the EtM operators
for the original scattering problem and the truncated PML problem established in this paper via
the decay estimate of the dyadic Green's function.

Our method can be extended to other electromagnetic scattering problems such as scattering
by inhomogeneous media or bounded elastic bodies as well as scattering in a two-layered medium.
It is also interesting to study the spherical and Cartesian PML methods for time-domain
elastic scattering problems, which is more challenging due to the existence of shear and compressional
waves with different wave speeds. We hope to report such results in the near future.

\section*{Acknowledgments}

This work was partly supported by the NNSF of China grants 11771349 and 91630309.
The first author was also partly supported by the National Research Foundation of Korea
(NRF-2020R1I1A1A01073356).

\end{document}